\documentclass[12pt]{article}
\usepackage[a4paper,margin=27mm]{geometry}
\usepackage[T1]{fontenc}
\usepackage{lmodern}
\usepackage{amsmath,amssymb,amsthm,mathtools}
\usepackage{graphicx}
\graphicspath{{images/}}
\usepackage[hidelinks]{hyperref}
\allowdisplaybreaks
\numberwithin{equation}{section}
\newtheorem{theorem}{Theorem}[section]
\newtheorem{lemma}[theorem]{Lemma}
\newtheorem{proposition}[theorem]{Proposition}
\newtheorem{corollary}[theorem]{Corollary}
\newtheorem{conjecture}[theorem]{Conjecture}
\newtheorem{claim}[theorem]{Claim}
\newtheorem{definition}[theorem]{Definition}
\newtheorem{example}[theorem]{Example}
\newcommand{\F}{\mathcal F}
\newcommand{\G}{\mathcal G}
\newcommand{\cH}{\mathcal H}
\newcommand{\M}{\mathcal M}
\newcommand{\hht}{\mathcal T}
\newcommand{\Rr}{\mathcal R}
\title{The maximum size of intersecting $k$-graphs\\
with covering number $k$}
\author{Peter Frankl$^1$, Jian Wang$^2$, Yaoran Yang$^3$, Yutong Zhang$^4$\\[10pt]
$^{1}$R\'{e}nyi Institute, Budapest, Hungary\\[6pt]
$^{2,3,4}$School of Mathematics, Sichuan University, Chengdu, 610065, China.\\[6pt]
E-mail:   $^1$frankl.peter@renyi.hu, $^2$wangjianmath01@scu.edu.cn\\[6pt]
 $^3$yangyaoran@stu.scu.edu.cn, $^4$yutongzhang@stu.scu.edu.cn
}
\date{}
\begin{document}
\maketitle

\begin{abstract}
Let $m(k)$ be the maximum size of an intersecting $k$-graph with
covering number $k$. In the present paper, it is proved that for all $k>e^{32}$,
\[
 m(k)\ge k^k\exp\bigl(-50\sqrt{k\log k}\bigr).
\]
Together with Zakharov's upper bound \cite{Z}, this implies
$m(k)=k^k\exp\bigl(-k^{1/2+o(1)}\bigr)$.
\end{abstract}

\section{Introduction}
A collection of distinct $k$-sets is called a \emph{$k$-graph}.
We use $V(\F)$ to denote the vertex set of a $k$-graph $\F$, that is, $V(\F) =\cup_{F\in \F}F$. 
A $k$-graph $\F$ is called \emph{intersecting} if
$F\cap F'\neq\varnothing$ for all $F,F'\in\F$.
A set $T$ is called a \emph{transversal} (or a \emph{cover}) of $\F$ if
$F\cap T\neq\varnothing$ for all $F\in\F$.
Define the \emph{covering number}
\[
 \tau(\F):=\min\{|T|:T\text{ is a transversal \ of }\F\}.
\]
If $\F$ is intersecting then every $F\in\F$ is a transversal
whence $\tau(\F)\leq k$.
All logarithms are natural.

The following result can be extracted from the seminal
paper of Erd\H{o}s and Lov\'asz~\cite{EL}.
\begin{theorem}\label{intro:EL-upper}
Suppose that $\F$ is an intersecting\ $k$-graph, $\tau(\F)=k$. Then
\[
 |\F|\leq k^k.
\]
\end{theorem}

Let us show the short proof. First a definition.
For a set $A$, define $\F(A)=\{F\setminus A:A\subset F\in\F\}$ and set
\[
 \Delta_\ell(\F)=\max\{|\F(A)|:|A|=\ell\},\qquad
 \Delta_0(\F)=|\F|.
\]

\begin{claim}\label{intro:degree-growth}
$\Delta_{\ell+1}(\F)\geq\Delta_\ell(\F)/k$ for $k>\ell\geq0$.
\end{claim}
\begin{proof}
Fix $A$, $|A|=\ell$ satisfying $|\F(A)|=\Delta_\ell(\F)$.
Since $|A|<\tau(\F)$, there exists $F\in\F$ with $A\cap F=\varnothing$.
As $\F$ is intersecting, $G\cap F\neq\varnothing$ for all $G\in\F(A)$.
Hence
\[
 \sum_{x\in F}|\F(A\cup\{x\})|\geq|\F(A)|=\Delta_\ell(\F).
\]
Consequently
\[
 \Delta_{\ell+1}(\F)
 \geq\max_{x\in F}|\F(A\cup\{x\})|
 \geq\Delta_\ell(\F)/k. \qedhere
\]
\end{proof}

As $\Delta_k(\F)=1$, $\Delta_{k-1}(\F)\leq k$,
$\Delta_{k-2}(\F)\leq k^2$, $\ldots$, $\Delta_0(\F)\leq k^k$
follows. 

In view of Theorem~\ref{intro:EL-upper} we can define
\[
 m(k)=\max\{|\F|:\F\text{ is an intersecting \ $k$-graph with }\tau(\F)=k\}.
\]
$m(1)=1$, $m(2)=3$ are trivial and $m(3)=10$ was proved by
Lov\'asz~\cite{lovasz}.
For $k=4$, there is already a large gap between the lower
and upper bounds (cf.~\cite{FOT,FW}):
\[
 42\leq m(4)\leq64.
\]

Erd\H{o}s and Lov\'asz~\cite{EL} proved that
$m(k)\geq\lfloor k!(e-1)\rfloor$.
Lov\'asz~\cite{lovasz} conjectured that this is the exact bound.
In \cite{FOT}, Lov\'asz's conjecture was disproved for $k\geq4$,
and the lower bound for $m(k)$ was improved to $(k/2)^{k-1}$.

The Erd\H{o}s--Lov\'asz lower bound follows from the next construction,
which will also be used in the proof of Corollary~\ref{des:product}.

\begin{example}[Erd\H{o}s--Lov\'asz {\cite{EL}}]\label{ex:EL}
Let $E_1,\ldots,E_k$ be pairwise disjoint sets with $|E_i|=i$,
and set
\[
 \mathcal E=\bigcup_{i=1}^k
 \{E_i\cup\{x_{i+1},\ldots,x_k\}:x_j\in E_j\ (i<j\le k)\}.
\]
Then it can be checked that  $\mathcal E$ is an intersecting $k$-graph with $ \tau(\mathcal E)=k$ and 
\[
 |\mathcal E|=\sum_{i=1}^k\frac{k!}{i!}
             =\lfloor(e-1)k!\rfloor.
\]
\end{example}

 In  \cite{FOT}, the first author, Ota and Tokushige proposed the following conjecture.

\begin{conjecture}[Frankl--Ota--Tokushige {\cite{FOT}}]
\label{conj:FOT}
There is an absolute constant $c$, with $1/2<c<1$, such that
\[
 m(k)<(ck)^k
\]
for all sufficiently large integers $k$.
\end{conjecture}

The upper bound $k^k$ of Erd\H{o}s--Lov\'asz was successively improved by
Tuza~\cite{T}, Arman and Retter~\cite{AR}, and the first
author~\cite{Fr}; see also Cherkashin~\cite{Ch} for the related
three-chromatic problem. The current best upper bound is due to Zakharov~\cite{Z}, who proved that
\begin{equation}\label{intro:Zakharov}
 m(k)\le k^k\exp\bigl(-k^{1/2+o(1)}\bigr).
\end{equation}
The main result of this paper is the following.
\begin{theorem}\label{main:lower}
For every integer $k$ with $\log k\geq32$,
\[
 m(k)\ge k^k\exp\bigl(-50\sqrt{k\log k}\bigr).
\]
\end{theorem}

Together with \eqref{intro:Zakharov}, Theorem~\ref{main:lower} implies 
$m(k)=k^k\exp\bigl(-k^{1/2+o(1)}\bigr)$, which disproves Conjecture~\ref{conj:FOT} for all sufficiently large $k$.

The construction is based on the following theorem, which provides
an intersecting $k$-graph with many minimum covers.

For $0<v<k$, let $\hht^{(v)}(\G)$ denote the family of all transversals of $\G$ of size $v$. 

\begin{theorem}\label{main:many-covers}
For every integer $k$ with $\log k\geq32$, there are integers
$0<b<v<k$ satisfying
\begin{equation}\label{intro:cover-bounds}
 b\geq k- 37\sqrt{k\log k},\qquad
 v\geq k - 20\sqrt{k/\log k},
\end{equation}
and an intersecting $k$-graph $\G$ on a finite set $X$
with $\tau(\G)=v$ such that $|\hht^{(v)}(\G)|\geq b^v$ and $\tau(\hht^{(v)}(\G))=k$.
\end{theorem}

\section{The construction}
\label{sec:designated}

In Section~\ref{sec:parameter-choice}, we deduce
Theorem~\ref{main:lower} from Theorem~\ref{main:many-covers}.
The remaining subsections prove Theorem~\ref{main:many-covers}.
We construct an auxiliary uniform hypergraph $\cH$ from a fixed collection
of edge-disjoint perfect matchings of the complete graph $K_{2n}$.
The vertices of $\cH$ represent classes of matching edges.
Adding new vertices of degree one provides the hypergraph $\mathcal L$ of
Corollary~\ref{des:auxiliary} in Section~\ref{sec:auxiliary}.
We then use $\mathcal L$ and disjoint blocks $B_1,\ldots,B_v$
to construct a random intersecting $k$-graph $\G$ on
$X=V(\mathcal L)\cup(\bigcup_iB_i)$ in Section~\ref{sec:intersecting}.
Sections~\ref{sec:small-covers} and \ref{sec:covering} show that
$\tau(\G)=v$ with positive probability and complete the proof.

\subsection{ Proof of Theorem~\ref{main:lower} assuming
Theorem~\ref{main:many-covers}}
\label{sec:parameter-choice}

Let us first prove a lemma.

\begin{lemma}\label{des:completion}
Let $0<v<k$ be integers. Let $\G$ be an intersecting $k$-graph
satisfying $\tau(\G)=v$ and $\tau(\hht^{(v)}(\G))=k$.
Let further $\mathcal E$ be an intersecting $(k-v)$-graph
satisfying $\tau(\mathcal E)=k-v$.
Suppose that the vertex sets $V(\G)$ and $V(\mathcal E)$ are disjoint.
Define
\[
 \cH=\{T\cup E:T\in\hht^{(v)}(\G),\ E\in\mathcal E\}.
\]
Then $\F:=\G\cup\cH$ is an intersecting $k$-graph,
$\tau(\G\cup\cH)=k$, and
\[
 |\G\cup\cH|>|\cH|=|\hht^{(v)}(\G)|\,|\mathcal E|.
\]
\end{lemma}

\begin{proof}
Since both $\G$ and $\mathcal E$ are intersecting, to prove that
$\G\cup\cH$ is intersecting as well we only need to show that
$G\cap H\neq\varnothing$ for $G\in\G$, $H\in\cH$.
This is immediate from $G\cap T\neq\varnothing$ for all
$T\in\hht^{(v)}(\G)$.

To prove $\tau(\G\cup\cH)=k$, let $R$ be a transversal of $\F$.
If $R$ is a transversal of $\hht^{(v)}(\G)$, then $|R|\geq k$
is immediate. Otherwise fix $T\in\hht^{(v)}(\G)$ with
$T\cap R=\varnothing$. Then
\[
 R\cap E=R\cap(T\cup E)\neq\varnothing
 \quad\text{for all }E\in\mathcal E.
\]
Thus $|R\cap V(\mathcal E)|\geq\tau(\mathcal E)=k-v$.
Also, $|R\cap V(\G)|\geq\tau(\G)=v$.
Hence $|R|\geq k$. Since $\F$ is an intersecting $k$-graph,
$\tau(\F)=k$.

The disjointness of $V(\G)$ and $V(\mathcal E)$ implies
$|\cH|=|\hht^{(v)}(\G)|\,|\mathcal E|$ and
$\G\cap\cH=\varnothing$, so $|\G\cup\cH|>|\cH|$.
\end{proof}

\begin{corollary}\label{des:product}
Let $b,v$ be the integers given by Theorem~\ref{main:many-covers}.
Then there is an intersecting $k$-graph with covering number $k$
and at least $b^v(k-v)!$ edges.
\end{corollary}
\begin{proof}
Let $\G$ be given by Theorem~\ref{main:many-covers} and
$\mathcal E$ by Example~\ref{ex:EL}, with uniformity $k-v$
and vertex set disjoint from $V(\G)$.
Then $\F=\G\cup\cH$ satisfies all requirements by
Lemma~\ref{des:completion}, since
$|\hht^{(v)}(\G)|\geq b^v$ and $|\mathcal E|\geq(k-v)!$.
\end{proof}

\begin{proof}[Proof of Theorem~\ref{main:lower}]
Let $k$ be an integer with $\log k\geq32$, and choose $b,v$
as in Theorem~\ref{main:many-covers}.
By Corollary~\ref{des:product} and
$(k-v)!\geq((k-v)/e)^{k-v}$, we have
\[
 \log m(k)\geq v\log b+(k-v)\log\frac{k-v}{e}.
\]
Since $0<b<v<k$, we have $\log(b/k)<0$ and hence
$v\log(b/k)\geq k\log(b/k)$. Therefore,
\[
\begin{aligned}
 v\log b+(k-v)\log\frac{k-v}{e}
 &=k\log k+v\log\frac bk
   -(k-v)\log\frac{ek}{k-v}\\
 &\geq k\log k+k\log\frac bk
   -(k-v)\log\frac{ek}{k-v}.
\end{aligned}
\]
Applying $ \log(1-x)
\geq -x-\frac{x^2}{2(1-x)}$ for $0<x<1$ with $x=(k-b)/k$, we infer that 
\[
\begin{aligned}
 k\log\frac bk=k\log\left(1-\frac{k-b}{k}\right)\geq-(k-b)
   -\frac{k\bigl((k-b)/k\bigr)^2}
          {2\bigl(1-(k-b)/k\bigr)}=-(k-b)-\frac{(k-b)^2}{2b}.
\end{aligned}
\]
Consequently,
\[
 \log m(k)\geq k\log k-(k-b)
 -(k-v)\log\frac{ek}{k-v}
 -\frac{(k-b)^2}{2b}.
\]

Since $\log k\geq32$, \eqref{intro:cover-bounds} implies $b\geq k/2$ 
and $e\leq20\sqrt{k/\log k}\leq k$.
The function $x\log(ek/x)$ is increasing on $(0,k]$, so
\eqref{intro:cover-bounds} yields
\[
 (k-v)\log\frac{e k}{k-v}
 \leq20\sqrt{\frac{k}{\log k}}\,
       \log\frac{e k}{20\sqrt{k/\log k}}
 \leq12\sqrt{k\log k}.
\]
For the last inequality, we used
\[
 \log\frac{e k}{20\sqrt{k/\log k}}
 =\frac12\log k+\frac12\log\log k+1-\log20
 \leq\frac35\log k,
\]
since $\log\log k\leq\frac15\log k$ and $e<20$.
Also,
\[
 \frac{(k-b)^2}{2b}\leq1600\log k\leq\sqrt{k\log k}.
\]
Consequently,
\[
 \log m(k)\geq k\log k-(37+12+1)\sqrt{k\log k}
             = k\log k-50\sqrt{k\log k}.
\]
Thus the theorem follows.
\end{proof}

\subsection{An auxiliary hypergraph}
\label{sec:auxiliary}

Define the {\it degree of $x$} as $d_{\F}(x)=|\{F\in\F:x\in F\}|$ and the {\it maximum degree of $\F$} as 
$\Delta(\F)=\max_x d_{\F}(x)=\Delta_1(\F)$.
Let  $[N]=\{1,\ldots,N\}$ and let $\mathbf1_A$ be $1$ if the
condition $A$ holds and $0$ otherwise.

We use the following standard consequence of
\cite[Lemma~1 and Theorem~4]{Hoeffding}.

\begin{lemma}[Chernoff--Hoeffding bound \cite{Hoeffding}]\label{des:chernoff}
Let $Z$ be a real-valued random variable with
$\mathbb E2^Z<\infty$. Then
\begin{equation}\label{des:chernoff-general}
 \Pr(Z\geq0)\leq\mathbb E2^Z.
\end{equation}
\begin{enumerate}
\item[(i)] Let $w_1,\ldots,w_N$ be fixed real numbers,
let $B$ be chosen uniformly at random from $\binom{[N]}m$,
and let $Z=\sum_{j\in B}w_j$.
Then for every $t>0$,
\begin{equation}\label{des:chernoff-sampling}
 \Pr(Z\geq0)
 \leq \mathbb E e^{tZ}
 \leq\exp\left\{\frac mN\sum_{j=1}^N(e^{tw_j}-1)\right\}.
\end{equation}
If $w_j\in[0,1]$ for every $j\in[N]$, then for every
$a\geq0$ and $t>0$,
\[
 \Pr(Z\geq a)\leq\exp\{-ta+(e^t-1)\mathbb EZ\}.
\]
\item[(ii)] Let $Z_1,\ldots,Z_N$ be independent random variables
taking values in $[0,1]$, and put $Z=\sum_{j=1}^N Z_j$.
Then for every $a\geq0$ and $t>0$,
\begin{equation}\label{des:chernoff-tail}
 \Pr(Z\geq a)\leq\exp\{-ta+(e^t-1)\mathbb EZ\}.
\end{equation}
\end{enumerate}
\end{lemma}

\begin{proposition}[Kirkman~\cite{kirkman1847}]
\label{prop-kirkman}
For every $n\geq1$, the complete graph $K_{2n}$ can be partitioned into
$2n-1$ perfect matchings.
\end{proposition}

\begin{figure}[h]
\centering
\includegraphics[width=0.5\textwidth]{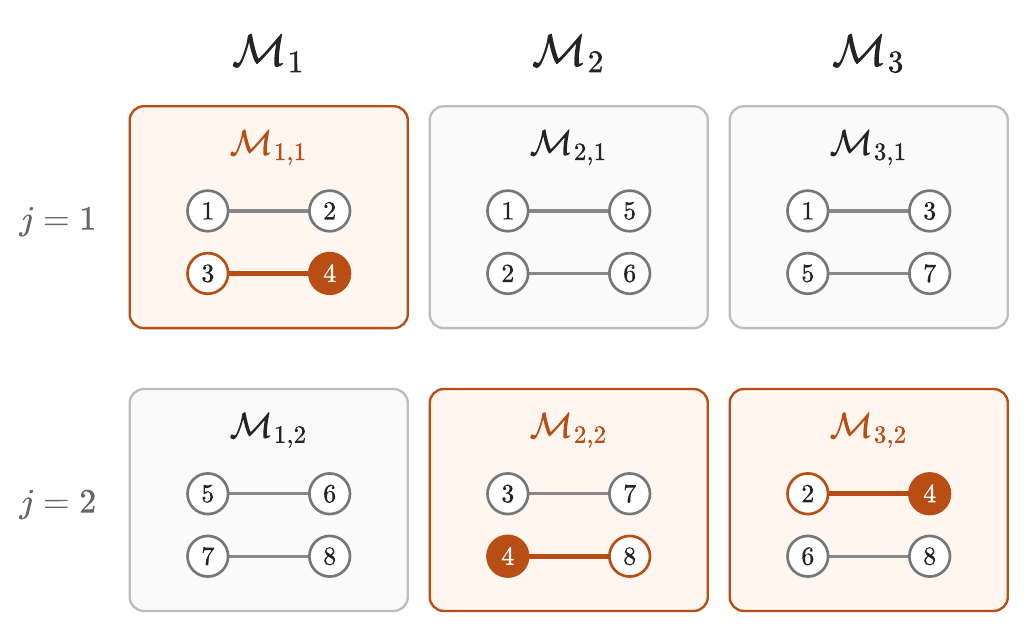}\hspace{50pt}
\includegraphics[width=0.35\textwidth]{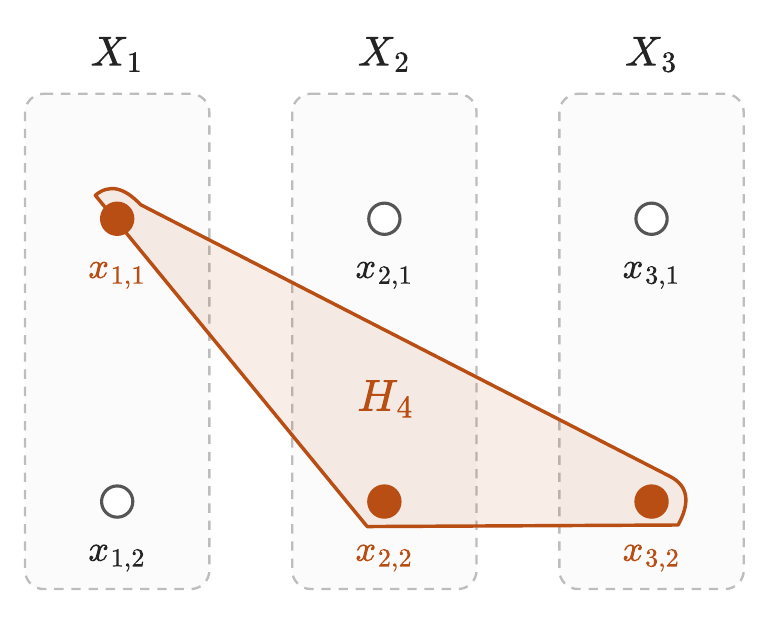}
\caption{The case $2n=8$, $r=3$, $q=2$ and $H_4=\{x_{1,1},x_{2,2},x_{3,2}\}$.}\label{figure-1}
\end{figure}

\begin{definition}\label{des:random-definition}
Let $1\leq q\leq n$ and $1\leq r\leq2n-1$.
Fix $r$ edge-disjoint perfect matchings $\M_1,\ldots,\M_r$
of $K_{2n}$ from Proposition~\ref{prop-kirkman}.
Independently for each $i\in[r]$,
choose an ordered partition 
\[
 \M_i=\M_{i,1}\mathbin{\dot\cup}\cdots
             \mathbin{\dot\cup}\M_{i,q}
\]
uniformly at random, whose class sizes differ by at most one.
For each class $\M_{i,j}$, introduce a new vertex $x_{i,j}$.
Thus $X_i=\{x_{i,j}:j\in[q]\}$, $i=1,2,\ldots,r$, are pairwise disjoint $q$-sets. Set
$V(\M_{i,j})=\bigcup_{P\in\M_{i,j}}P$.
For $p\in[2n]$ and $\ell\in[r]$, let $g_\ell(p)$ be the unique
$j\in[q]$ such that $p\in V(\M_{\ell,j})$, and define
\[
 H_p=\{x_{\ell,g_\ell(p)}:\ell\in[r]\},\ 
 \cH=\{H_p:p\in[2n]\}.
\]
The ground set of $\cH$ is $X_1\cup\cdots\cup X_r$;
the elements of $[2n]$ index its edges.
Each $H_p$ contains one vertex from each $X_\ell$, recording the
class containing the matching edge through $p$ in $\M_\ell$.
See Figure~\ref{figure-1} for an illustration.
Thus $\cH$ is a random $r$-partite $r$-graph on $X_1,\ldots,X_r$.
Since each $\M_{i,j}$ contains at most $\lceil n/q\rceil$ edges, we infer that $|V(\M_{i,j})|\leq 2\lceil n/q\rceil$. It follows that $x_{i,j}$ are contained in at most $2\lceil n/q\rceil$ edges of $\cH$. Thus
$\Delta(\cH)\leq 2\lceil n/q\rceil$.
\end{definition}

\begin{proposition}\label{des:random-core}
Let $m$ be a positive integer.
If
\begin{equation}\label{des:core-conditions}
 r/q\geq64\log(2n)\mbox{ and } m\geq2r/q,
\end{equation}
then the following hold simultaneously with probability at least $1/2$:
\begin{enumerate}
\item[(i)] $|H_i\cap H_j|\leq m$ for distinct $i,j\in[2n]$;
\item[(ii)] for every $i\in[2n]$ and $U\subseteq H_i$ with
$|U|\leq r/4$,
\[
 \left|\left\{j\in[2n]\setminus\{i\}:
 |U\cap H_j|\geq|H_i\cap H_j|/2\right\}\right|<16q.
\]
\item[(iii)] There are sets
$P_i\subseteq[2n]\setminus\{i\}$ of size $r$, with
$j\in P_i$ if and only if $i\in P_j$. Moreover, there are vertices
$a_{ij}=a_{ji}\in H_i\cap H_j$ such that
$H_i=\{a_{ij}:j\in P_i\}$.
\end{enumerate}
\end{proposition}

Note that for any $\ell\in [r]$ and $p\in [q]$, we have $a_{ij} =a_{i'j'}=x_{\ell,p}$ for any $(i,j),(i',j')\in \M_{\ell, p}$. 

\begin{proof}
For any distinct $i,j\in [2n]$,
\[
 |H_i\cap H_j|
 =\sum_{\ell=1}^r\mathbf1_{\{g_\ell(i)=g_\ell(j)\}}.
\]
 Since $\M_1,\ldots,\M_r$ are pairwise edge-disjoint, at most one of
$\M_1,\ldots,\M_r$ contains $(i,j)$. For any $\M_\ell$ with $(i,j)\notin \M_{\ell}$, the probability that $g_\ell(i)=g_\ell(j)$ is at most 
\[
\frac{\binom{n-2}{\lceil n/q\rceil-2}}
     {\binom{n-1}{\lceil n/q\rceil-1}}
 =\frac{\lceil n/q\rceil-1}{n-1}\leq\frac1q.
\]
Without loss of generality, assume that $\M_\ell$ does not contain  $(i,j)$ for $\ell=1,2,\ldots, r-1$. Then 
\[
\mathbb E\bigl[\sum_{\ell=1}^{r-1}\mathbf1_{\{g_\ell(i)=g_\ell(j)\}}\bigr]<\frac{r}{q}. 
\]
Applying \eqref{des:chernoff-tail} with $a=m$ and $t=\log2$,
\[
 \Pr\bigl( |H_i\cap H_j| \geq m+1\bigr) \leq  \Pr\bigl( \sum_{\ell=1}^{r-1}\mathbf1_{\{g_\ell(i)=g_\ell(j)\}} \geq m\bigr) \leq \exp\{-m\log2+r/q\}.
\]
Since $m\geq 2r/q$ and $r/q\geq 64\log(2n)$ imply $-m\log2+r/q\leq -r/(4q)\leq -16\log(2n)$, 
by the union bound we obtain that
\begin{equation}\label{des:property-i}
\begin{aligned}
 \Pr\bigl(\textup{(i) fails}\bigr)
 \leq4n^2\exp\{-m\log2+r/q\}\leq4n^2e^{-16\log(2n)}=
 (2n)^{-14}<1/4.
\end{aligned}
\end{equation}

For (ii), recall that $\cH$ is an $r$-partite $r$-graph on partite sets $X_1,X_2,\ldots,X_r$ with $|X_i|=q$. For any $Q\subseteq [r]$, define 
$X_Q= \cup_{\ell\in Q}X_\ell$.  For $i\in[2n]$, $Q\subseteq[r]$ with
$|Q|\leq r/4$, and $J\in\binom{[2n]\setminus\{i\}}{16q}$, define the bad event $E_{i,Q,J}$ as the event that 
\begin{align}\label{ineq-2.1-1}
 |H_i\cap H_j\cap X_Q|\geq|H_i\cap H_j|/2 \mbox{ holds for all  }j \in J. 
\end{align}
Obviously, (ii) holds if and only if all the bad events $E_{i,Q,J}$ do not happen for $i\in [2n]$, $Q\subseteq[r]$ with
$|Q|\leq r/4$, and $J\in\binom{[2n]\setminus\{i\}}{16q}$. 

For $\ell\in [r]$, define 
\[
 Y_\ell=|\{j\in J:g_\ell(i)=g_\ell(j)\}| =\sum_{j\in J} |H_i\cap H_j\cap X_\ell|.
\]
By the definition of $Y_\ell$, we have
\[
 \sum_{\ell\in Q}Y_\ell=\sum_{j\in J}|H_i\cap H_j\cap X_Q|,
 \qquad
 \sum_{\ell\notin Q}Y_\ell
 =\sum_{j\in J}|H_i\cap H_j\setminus X_Q|.
\]
If \eqref{ineq-2.1-1} holds, then by the definition of $Y_\ell$ we have
\begin{align}\label{ineq-2.1-2}
 \sum_{\ell\in Q}Y_\ell=\sum_{j\in J}|H_i\cap H_j\cap X_Q|\geq\sum_{j\in J}|H_i\cap H_j\setminus X_Q|  =\sum_{\ell\notin Q}Y_\ell.
\end{align}
Define the event $E_{i,Q,J}'$  as the event that \eqref{ineq-2.1-2} holds. Then $\Pr(E_{i,Q,J})\leq \Pr(E_{i,Q,J}')$. 

Note that
$Y_\ell=\sum_{j\in J}\mathbf1_{\{g_\ell(i)=g_\ell(j)\}}$.
If $(i,j)\in\M_\ell$, then $g_\ell(i)=g_\ell(j)$
with probability one.
If $(i,j)\notin\M_\ell$, then
\[
 \frac1q-\frac{2}{n-1}
 \leq \Pr\bigl(g_\ell(i)=g_\ell(j)\bigr)
 \leq\frac1q.
\]
Indeed, conditional on its size,
$\M_{\ell,g_\ell(i)}$ is uniformly distributed over the subsets
of $\M_\ell$ of that size containing the matching edge through $i$,
and its size is either $\lfloor n/q\rfloor$ or $\lceil n/q\rceil$.  Assume that $\M_\ell =\{P_{\ell,1},P_{\ell,2},\ldots,P_{\ell,n}\}$ with $i\in P_{\ell,1}$ for every
$\ell\in[r]$. Let $w_{\ell,p}=|P_{\ell,p}\cap J|$ for $\ell\in Q$ and $w_{\ell,p}=-|P_{\ell,p}\cap J|$ for $\ell\notin Q$. Then
\begin{align}\label{ineq-2.1-3}
\sum_{\ell\in Q}Y_\ell-\sum_{\ell\notin Q}Y_\ell = \sum_{\ell \in [r]} \sum_{p=1}^n w_{\ell,p}
       \mathbf1_{\{P_{\ell,p}\in\M_{\ell,g_\ell(i)}\}}. 
\end{align}

Conditional on $|\M_{\ell,g_\ell(i)}|$, the other
$|\M_{\ell,g_\ell(i)}|-1$ edges form a uniform sample from
$P_{\ell,2},\ldots,P_{\ell,n}$.
Applying the second inequality of \eqref{des:chernoff-sampling}
to $w_{\ell,2},\ldots,w_{\ell,n}$ with $t=\log2$, and including
the always selected edge $P_{\ell,1}$, we obtain that 
\[
\begin{aligned}
 &\quad \ \log\mathbb E\left[
 2^{\sum_{p=1}^n w_{\ell,p}
       \mathbf1_{\{P_{\ell,p}\in\M_{\ell,g_\ell(i)}\}}}
 \,\middle|\,|\M_{\ell,g_\ell(i)}|\right]\\[3pt]
 &\leq w_{\ell,1}\log2
 +\frac{|\M_{\ell,g_\ell(i)}|-1}{n-1}
       \sum_{p=2}^n(2^{w_{\ell,p}}-1)\\[3pt]
 &=\frac{|\M_{\ell,g_\ell(i)}|-1}{n-1}
       \sum_{p=1}^n(2^{w_{\ell,p}}-1)
 +w_{\ell,1}\log2
 -\frac{|\M_{\ell,g_\ell(i)}|-1}{n-1}
       (2^{w_{\ell,1}}-1).
\end{aligned}
\]

Since $J\in \binom{[2n]\setminus \{i\}}{16q}$, we infer that $|w_{\ell,1}| = |P_{\ell,1}\cap J|\leq 1$ and the number of $P_{\ell,1}$ 
with $|P_{\ell,1}\cap J|=1$ is at most $16q$. 
Note that $\frac{1}{q} - \frac{2}{n-1} \leq \frac{|\M_{\ell,g_\ell(i)}|-1}{n-1}\leq \frac{1}{q}$. We infer that 
\[
\frac{|\M_{\ell,g_\ell(i)}|-1}{n-1}
       \sum_{p=1}^n(2^{w_{\ell,p}}-1) \leq \frac1q\sum_{p=1}^n(2^{w_{\ell,p}}-1)+\frac{16q}{n-1}\mathbf1_{\{\ell\notin Q\}}. 
\]
If $\ell\in Q$, then $w_{\ell,1}\log2
 -\frac{|\M_{\ell,g_\ell(i)}|-1}{n-1}
       (2^{w_{\ell,1}}-1)\leq |P_{\ell,1}\cap J|$. If $\ell\notin Q$, then $w_{\ell,1}\log2
 -\frac{|\M_{\ell,g_\ell(i)}|-1}{n-1}
       (2^{w_{\ell,1}}-1)\leq 0$,
since
\[
 -\log2+\frac{|\M_{\ell,g_\ell(i)}|-1}{2(n-1)}
 \leq-\log2+\frac12<0.
\]
Thus, 
\[
\log\mathbb E\left[
 2^{\sum_{p=1}^n w_{\ell,p}
       \mathbf1_{\{P_{\ell,p}\in\M_{\ell,g_\ell(i)}\}}}
 \,\middle|\,|\M_{\ell,g_\ell(i)}|\right] \leq \frac1q\sum_{p=1}^n(2^{w_{\ell,p}}-1)+\frac{16q}{n-1}\mathbf1_{\{\ell\notin Q\}} +|P_{\ell,1}\cap J|\mathbf1_{\{\ell\in Q\}}. 
\]
The bound is uniform in $|\M_{\ell,g_\ell(i)}|$. Averaging and using
\eqref{ineq-2.1-3} and the independence of $Y_1,\ldots,Y_r$,
which depend on the independently chosen partitions of
$\M_1,\ldots,\M_r$, respectively, we obtain
\begin{equation}\label{des:layer-mgf}
\begin{aligned}
 \log\mathbb E\,2^{\sum_{\ell\in Q}Y_\ell
                         -\sum_{\ell\notin Q}Y_\ell}
 &\leq \sum_{\ell\in [r]}\frac1q\sum_{p=1}^n(2^{w_{\ell,p}}-1)+\frac{16q}{n-1}\sum_{\ell\in [r]}\mathbf1_{\{\ell\notin Q\}} +\sum_{\ell\in [r]}|P_{\ell,1}\cap J|\mathbf1_{\{\ell\in Q\}}\\
 &\leq \sum_{\ell\in [r]}\frac1q\sum_{p=1}^n(2^{w_{\ell,p}}-1)+\frac{16rq}{n-1} +16q. 
\end{aligned}
\end{equation}
Since each $\M_\ell$ is a perfect matching,
$\sum_{p=1}^n|P_{\ell,p}\cap J|=16q$. Hence
\[
 \sum_{p=1}^n(2^{w_{\ell,p}}-1)
 =
 \begin{cases}
 \displaystyle 16q+\sum_{p=1}^n
       \mathbf1_{\{P_{\ell,p}\subseteq J\}},
       &\ell\in Q,\\[5pt]
 \displaystyle-8q+\frac14\sum_{p=1}^n
       \mathbf1_{\{P_{\ell,p}\subseteq J\}},
       &\ell\notin Q.
 \end{cases}
\]
Since the matchings are edge-disjoint,
\[
 \sum_{\ell=1}^r\sum_{p=1}^n
 \mathbf1_{\{P_{\ell,p}\subseteq J\}}\leq\binom{16q}{2}.
\]
Summing over $\ell$ and using $1/4\leq1$, we obtain
\[
 \frac1q\sum_{\ell=1}^r\sum_{p=1}^n(2^{w_{\ell,p}}-1)
 \leq16\left(|Q|-\frac{r-|Q|}{2}\right)
      +\frac1q\binom{16q}{2}.
\]
Substituting this estimate for
$\frac1q\sum_{\ell=1}^r\sum_{p=1}^n(2^{w_{\ell,p}}-1)$
into \eqref{des:layer-mgf}, we obtain
\[
\begin{aligned}
 \log\mathbb E\,2^{\sum_{\ell\in Q}Y_\ell
                         -\sum_{\ell\notin Q}Y_\ell}
 &\leq16\left(|Q|-\frac{r-|Q|}{2}\right)
       +\frac1q\binom{16q}{2}+\frac{16rq}{n-1}+16q\\
 &\leq-2r+\frac r2+\frac r4=-\frac{5r}{4}.
\end{aligned}
\]
Here $r\leq2n-1$ and \eqref{des:core-conditions} imply
$2n>64$ and $r\geq256q$; we also used $|Q|\leq r/4$.

Since $E_{i,Q,J}\subseteq E'_{i,Q,J}$, applying
\eqref{des:chernoff-general} to $\sum_{\ell\in Q}Y_\ell
                         -\sum_{\ell\notin Q}Y_\ell$ yields
\[
 \Pr(E_{i,Q,J})\leq\Pr(E'_{i,Q,J})
 =\Pr\left(\sum_{\ell\in Q}Y_\ell
          -\sum_{\ell\notin Q}Y_\ell\geq0\right)
 \leq e^{-5r/4}.
\]
If (ii) fails, some event $E_{i,Q,J}$ occurs.
There are at most $2n\,2^r\binom{2n-1}{16q}$ such events.
Using $\log2<3/4$, $(16q+1)\log(2n)\leq r/3$ and
the union bound, we obtain
\[
\begin{aligned}
 \Pr\bigl(\textup{(ii) fails}\bigr)
 &\leq2n\,2^r\binom{2n-1}{16q}e^{-5r/4}\\
 &\leq\exp\left\{r\log2+(16q+1)\log(2n)-\frac{5r}{4}\right\}\\
 &\leq e^{-r/6}<1/4.
\end{aligned}
\]
The bound $\Pr(\textup{(ii) fails})<1/4$ and
\eqref{des:property-i} show that (i) and (ii) hold simultaneously
with probability at least $1/2$.

For (iii), define
\[
 P_i=\{j\in[2n]\setminus\{i\}:
          (i,j)\in\M_1\cup\cdots\cup\M_r\}.
\]
Since the matchings are perfect and edge-disjoint, $|P_i|=r$,
and $j\in P_i$ if and only if $i\in P_j$.
For $j\in P_i$, let $\ell$ be the unique index such that
$(i,j)\in\M_\ell$, and set
\[
 a_{ij}=x_{\ell,g_\ell(i)}.
\]
The vertices $i,j$ belong to the same matching edge, so
$g_\ell(i)=g_\ell(j)$ and hence
$a_{ij}=a_{ji}\in H_i\cap H_j$.
For fixed $i$, the $r$ indices $j\in P_i$ correspond to the $r$ different
matchings, so the vertices $a_{ij}$ lie in different partite
sets and are distinct. Thus,
\[
 H_i=\{a_{ij}:j\in P_i\}.
\]
This proves (iii) for every choice of the partitions.
\end{proof}

Fix an integer $k$ with $\log k\geq32$, and set
\begin{equation}\label{des:choice}
\begin{gathered}
 h=2\left\lceil4\sqrt{k\log k}\right\rceil,\,
 s=\left\lceil\frac{h}{\log k}\right\rceil+2,\, 
 m=\left\lceil\frac{16h^2}{k}\right\rceil,\, 
 v=2\left\lfloor\frac{k-2s-16m}{2}\right\rfloor.
\end{gathered}
\end{equation}

\begin{lemma}\label{des:parameter-bounds}
Let $h,s,m,v$ be defined by \eqref{des:choice}.
These parameters are positive integers, with $h,v$ even,
$0<v<k$, and satisfy
\begin{equation}\label{des:conditions}
 h^2\geq32v\log v,\qquad
 \frac{32v}{h}\leq s\leq h,\qquad
 \frac{8h^2}{v}\leq m\leq\frac h{16}.
\end{equation}
Moreover,
\[
 k-v\leq20\sqrt{k/\log k}.
\]
\end{lemma}

\begin{proof}
By the definitions in \eqref{des:choice} and $\log k\geq32$,  we have
\[
\begin{gathered}
 8\sqrt{k\log k}\leq h\leq9\sqrt{k\log k},\,
 s\leq9\sqrt{\frac{k}{\log k}},\, 
 \frac{16h^2}{k}\leq m\leq\frac{32h^2}{k},\,
 16m+1\leq2\sqrt{\frac{k}{\log k}}.
\end{gathered}
\]
Hence,
\[
 0<k-v\leq2s+16m+1
 \leq20\sqrt{\frac{k}{\log k}}\leq k/2,
\]
which proves the bound on $k-v$ in \eqref{intro:cover-bounds}
and implies $v\geq k/2$.
The bounds on $h,s$ imply $s\leq h$.
Also, $h\leq9\sqrt{k\log k}\leq k/512$, so
$m\leq32h^2/k\leq h/16$.
The remaining inequalities in \eqref{des:conditions} follow from
\[
\begin{gathered}
 h^2\geq64k\log k\geq32v\log v,\quad
 sh\geq\frac{h^2}{\log k}\geq64k\geq32v\mbox{ and } 
 m\geq\frac{16h^2}{k}\geq\frac{8h^2}{v}. \qedhere
\end{gathered}
\]
\end{proof}

We fix $h,s,m,v$ as in \eqref{des:choice} for the rest of this section.
Lemma~\ref{des:parameter-bounds} ensures that
Corollary~\ref{des:auxiliary} applies to these values.

\begin{corollary}\label{des:auxiliary}
Let $h,s,v,m$ be positive integers, with $h,v$ even,
satisfying \eqref{des:conditions}.
Then there is a $(4h+2s)$-graph
$\mathcal L=\{L_1,\ldots,L_v\}$ on $V(\mathcal L)$,
with pairwise disjoint $2s$-sets $S_i\subseteq L_i$ whose
vertices have degree one.
There are sets $P_i\subseteq[v]\setminus\{i\}$ of size $4h$,
with $j\in P_i$ if and only if $i\in P_j$, and vertices
$a_{ij}=a_{ji}\in L_i\cap L_j$, distinct for fixed $i$, such that
\[
 L_i\setminus S_i=\{a_{ij}:j\in P_i\}.
\]
Moreover,
\[
 \Delta(\mathcal L)\leq h,\qquad
 \bigl|V(\mathcal L)\bigr|=4h\lceil v/h\rceil+2sv\leq v^2/4,
\]
and $|L_i\cap L_j|\leq m$ for distinct $i,j\in[v]$.
For every $i\in[v]$ and $U\subseteq L_i$ with
$|U\setminus S_i|\leq h$,
\[
 \left|\left\{j\in[v]\setminus\{i\}:
 |U\cap L_j|\geq|L_i\cap L_j|/2\right\}\right|<s.
\]
\end{corollary}

\begin{proof}
The inequalities in \eqref{des:conditions} imply
$h\leq v/32$ and $v\geq64$.
Apply Proposition~\ref{des:random-core} with
$2n=v$, $r=4h$ and $q=\lceil v/h\rceil$.
Here $4h\leq v-1$, 
$\lceil v/h\rceil\leq\min\{v/2,2v/h\}$ and
\[
\begin{aligned}
 \frac{4h}{\lceil v/h\rceil}
 \geq\frac{2h^2}{v}\geq64\log v,\qquad
 m \geq\frac{8h^2}{v}\geq\frac{8h}{\lceil v/h\rceil}.
\end{aligned}
\]
These verify \eqref{des:core-conditions}.
Fix a resulting $4h$-graph $\cH=\{H_1,\ldots,H_v\}$ and the sets $P_i$
and vertices $a_{ij}$ from Proposition~\ref{des:random-core}\textup{(iii)}.
Choose pairwise disjoint new $2s$-sets $S_i$ outside
$\bigcup\cH$ and define
\[
 L_i=H_i\mathbin{\dot\cup}S_i,\qquad
 \mathcal L=\{L_i:i\in[v]\}.
\]
Then $\mathcal L$ is a $(4h+2s)$-graph on $V(\mathcal L)$,
and each vertex of $S_i$ has degree one.
By Proposition~\ref{des:random-core}\textup{(iii)}, $|P_i|=4h$,
$j\in P_i$ if and only if $i\in P_j$, and
$a_{ij}=a_{ji}\in L_i\cap L_j$, distinct for fixed $i$, with
\[
 L_i\setminus S_i=\{a_{ij}:j\in P_i\}.
\]
Since $h$ is even,
\[
 \Delta(\mathcal L)
 \leq2\left\lceil\frac{v}{2\lceil v/h\rceil}\right\rceil\leq h.
\]
For $i\neq j$, $L_i\cap L_j=H_i\cap H_j$ and
$U\cap L_j=(U\cap H_i)\cap H_j$ whenever $U\subseteq L_i$.
Thus Proposition~\ref{des:random-core}\textup{(i)} gives
$|L_i\cap L_j|=|H_i\cap H_j|\leq m$ for $i\ne j$.
Since
\[
 \frac{s}{\lceil v/h\rceil}\geq\frac{sh}{2v}\geq16,
\]
we have $16\lceil v/h\rceil\leq s$.
If $|U\setminus S_i|\leq h$, then by applying Proposition~\ref{des:random-core} (ii) to $U\cap H_i$ we obtain that
\[
 \left|\left\{j\in[v]\setminus\{i\}:
 |U\cap L_j|\geq|L_i\cap L_j|/2\right\}\right|
 <16\lceil v/h\rceil\leq s.
\]
Since all matching classes are nonempty, we infer that 
\[
 \bigl|V(\mathcal L)\bigr|=4h\lceil v/h\rceil+2sv\leq4v+4h+2sv
 \leq5v+\frac{v^2}{8}\leq\frac{v^2}{4}.\qedhere
\]
\end{proof}

\subsection{A random intersecting $k$-graph $\mathcal G$ with many covers of size $v$}
\label{sec:intersecting}

Fix the hypergraph $\mathcal L$, the sets $S_i,P_i$, and the vertices
$a_{ij}$ provided by Corollary~\ref{des:auxiliary}, with the parameters
in \eqref{des:choice}.
For $U\subseteq L_i$, define
\[
 D_i(U)=\{j\in P_i:a_{ij}\in U\}
 \cup\{j\notin P_i\cup\{i\}:|U\cap L_j|\ge|L_i\cap L_j|/2\}.
\]

For any $U \subset \binom{L_i}{h+s}$ with $|U\setminus S_i|\le h$ and any $J\subset \binom{[v]\setminus\{i\}}{h+s}$, we call $(U,J)$  a {\it admissible pair for $i$} if $D_i(U)\subseteq J$. 

For $U\subseteq L_i$ with $|U\setminus S_i|\leq h$,
the vertices $a_{ij}$, $j\in P_i$, are distinct, so
\[
 |\{j\in P_i:a_{ij}\in U\}|
 \leq |U\setminus S_i|\leq h.
\]
By corollary~\ref{des:auxiliary} we have 
$|D_i(U)|\leq h+s-1$.
By \eqref{des:conditions}, $h+s\leq2h\leq v/16<v-2$.
Thus for any $U \subset \binom{L_i}{h+s}$ with $|U\setminus S_i|\le h$, one can extend $D_i(U)$ to some $J\subset \binom{[v]\setminus\{i\}}{h+s}$
such that  $(U,J)$ forms an admissible pair for $i$.

To define the blocks, set
\begin{equation}\label{des:parameters}
 b=k-4h-2s.
\end{equation}

\begin{lemma}\label{des:block-bounds}
The block size $b$ satisfies $0<b<v<k$ and
\begin{equation}\label{des:block-condition}
 0<v-b\leq4h-16m.
\end{equation}
Moreover, $k-b\leq37\sqrt{k\log k}$.
\end{lemma}
\begin{proof}
By the definition of $v$ in \eqref{des:choice} and $m\leq h/16$
from \eqref{des:conditions},
\[
 0<4h-16m-1\leq v-b\leq4h-16m,
\]
which proves \eqref{des:block-condition}.
Also, \eqref{des:choice} and $\log k\geq32$ imply
\[
 k-b=4h+2s
 \leq36\sqrt{k\log k}+18\sqrt{\frac{k}{\log k}}
 \leq37\sqrt{k\log k}\leq k/2.
\]
Thus $b>0$, and \eqref{des:block-condition} and
Lemma~\ref{des:parameter-bounds} imply $b<v<k$.
\end{proof}

By \eqref{des:conditions} and \eqref{des:block-condition},
$b\geq v-4h\geq4h+2s$, implying $|L_i|\leq |B_j|$. Now let us define a random intersecting $k$-graph by using the $(4h+2s)$-graph $\mathcal L$  and 
all admissible pairs $(U,J)$.

\begin{definition}[The random intersecting $k$-graph $\G$]
Take disjoint $b$-sets $B_1,\ldots,B_v$ outside $V(\mathcal L)$, and set
$X=V(\mathcal L)\cup\bigcup_iB_i$. Independently for each ordered pair $i\ne j$,
choose a uniform injection $f_{i,j}:L_i\to B_j$.
For every admissible $(U,J)$ and every bijection $g:U\to J$, define 
\[
 E_i(U,g)=B_i\cup(L_i\setminus U)
       \cup\{f_{i,g(y)}(y):y\in U\}
\]
By \eqref{des:parameters}, $|E_i(U,g)|=b+(3h+s)+(h+s)=k$,
as shown in Figure \ref{figure-2}. 
Define 
\[
 \G_i=\{E_i(U,g):(U,J)\text{ admissible for }i,\ 
                  g:U\to J\text{ bijective}\}.
\]
and  define $\G=\cup_{i=1}^{v}\G_i$.
\end{definition}

Note that $\G$ is a random $k$-graph on the fixed vertex set $\bigcup_{i=1}^v B_i 
   \cup V(\mathcal L)$.  With $\mathcal L$, the sets $S_i,P_i,B_i$, and the vertices $a_{ij}$
fixed, the randomness in $\G$ comes entirely from the injections
$f_{i,j}:L_i\to B_j$ for ordered pairs $(i,j)\in[v]^2$ with $i\ne j$,
chosen independently and uniformly from all such injections.

\begin{figure}[t]
\centering
\includegraphics[width=0.7\textwidth]{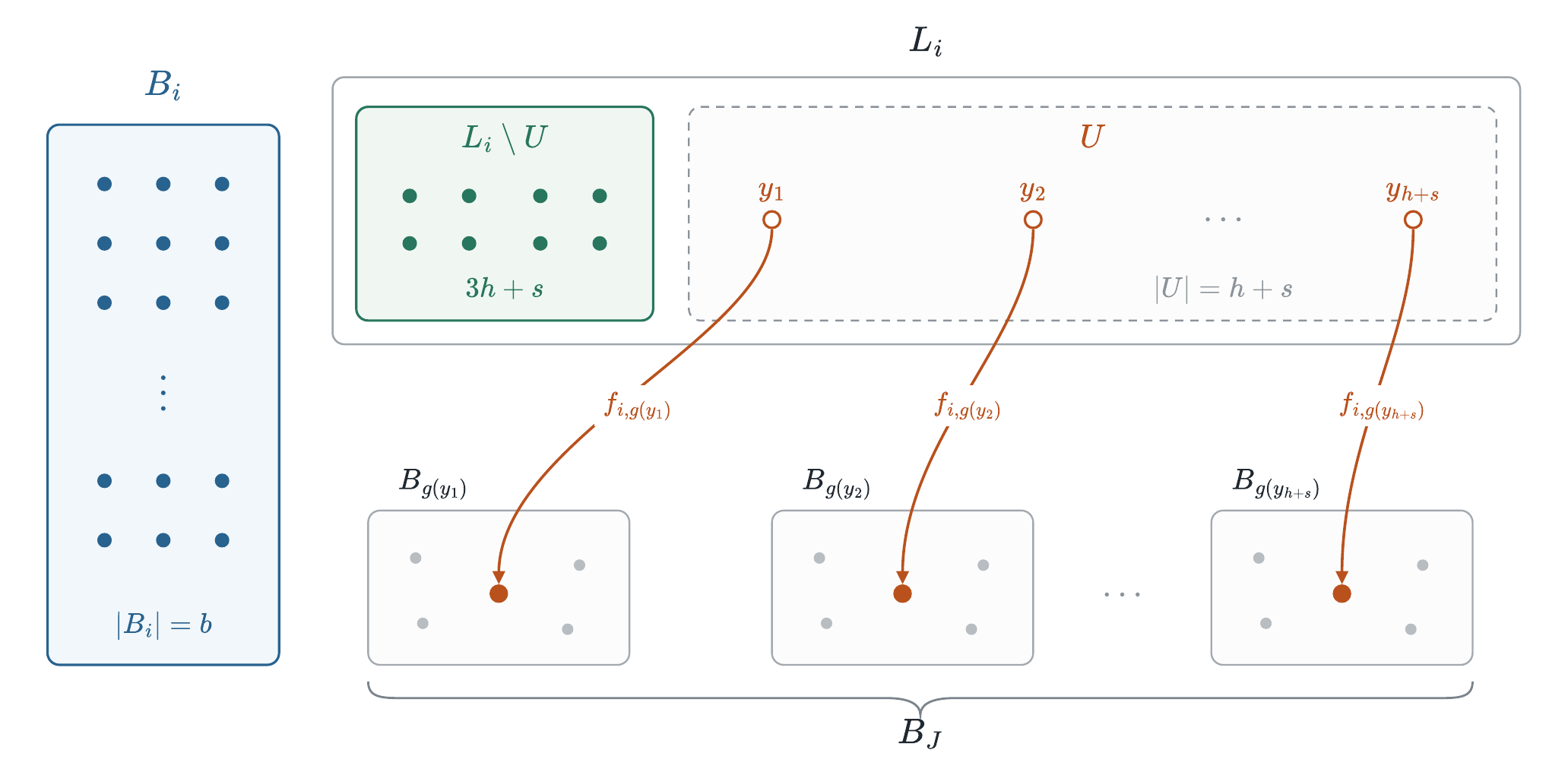}
\caption{The construction of $E_i(U,g)$.}\label{figure-2}
\end{figure}

Define $ \Rr=B_1\times B_2\times \ldots\times B_v$,  and for every $i\in [v]$, $y\in L_i$ define 
\[
 T_{i,y}
=\{y\}\cup
 \{f_{i,j}(y):j\in[v]\setminus\{i\}\}.
\]

\begin{claim}\label{des:intersection-covers}
For every choice of injections $f_{i,j}$ with $i\neq j$, $\G$ is intersecting with $\Rr\cup \{T_{i,y}\colon i\in [v],\ y\in L_i\}\subset \hht^{(v)}(\G)$. 
Moreover,  for each fixed $i$, $\{T_{i,y}\colon y\in L_i\}$ forms a matching of size $4h+2s$, which is disjoint from $B_i$.
\end{claim}

\begin{proof}
Note that each edge in $\G_i$ contains $B_i$. To show that $\G$ is intersecting, it is suffices to show that $\G_i$, $\G_j$ are cross-intersecting. 
Consider $E_i(U,g)\in\G_i$ and $E_j(U',g')\in\G_j$ with $i\ne j$,
where $(U,J)$ and $(U',J')$ are admissible for $i$ and $j$, respectively.
If $j\in J$ or $i\in J'$, the two edges intersect in $B_j$ or $B_i$,
respectively. Otherwise, by $D_i(U)\subseteq J$ and
$D_j(U')\subseteq J'$ we have
$j\notin D_i(U)$ and $i\notin D_j(U')$.

If $j\in P_i$, then $i\in P_j$ and
\[
 a_{ij}=a_{ji}\in(L_i\setminus U)\cap(L_j\setminus U'),
\]
so the two edges intersect.

If $j\notin P_i$, then $i\notin P_j$ and
\[
 |U\cap L_j|<\frac{|L_i\cap L_j|}{2},
 \qquad
 |U'\cap L_i|<\frac{|L_i\cap L_j|}{2}.
\]
Consequently,
\[
 |(L_i\cap L_j)\setminus(U\cup U')|
 \geq |L_i\cap L_j|-|U\cap L_j|-|U'\cap L_i|>0,
\]
so the two edges also intersect in this case.
Hence $\G$ is intersecting.

Each member of $\Rr$ contains one vertex from each $B_j$,
so it has size $v$ and intersects every edge of $\G$.
For $y\in L_i$,  $T_{i,y}$ has size $v$, since $y\in L_i$
and $f_{ij}(y)$ lies in  $B_j$ for all $j\neq i$. The set 
$T_{i,y}$ intersects $E_i(U,g)$ at $y$ if $y\notin U$, and at
$f_{i,g(y)}(y)$ if $y\in U$. Thus $T_{i,y}$ intersects every edge of $\G_i$. For $j\ne i$, $f_{i,j}(y)\in T_{i,y}\cap B_j$, so
$T_{i,y}$ intersects every edge of $\G_j$.
Therefore $T_{i,y}$ is a cover of $\G$. 

For fixed $i$ and distinct $y,y'\in L_i$,  by injectivity we infer that 
$f_{i,j}(y)\ne f_{i,j}(y')$ for every $j\ne i$.
Since $V(\mathcal L),B_1,\ldots,B_v$ are pairwise disjoint, we have
$T_{i,y}\cap T_{i,y'}=\varnothing$, and every $T_{i,y}$ is disjoint from $B_i$.
\end{proof}

\subsection{A necessary condition for $\tau(\mathcal G)<v$}
\label{sec:small-covers}

\begin{definition}\label{des:random-bipartite}
Fix $W\subseteq V(\G)$ with $|W|<v$, and $i\in[v]$ such that
$W\cap B_i=\varnothing$ and $|W\cap L_i|\leq h$.
Let $(U,J)$ be admissible for $i$, with $W\cap L_i\subseteq U$.
For a group of injections $f_{i,j}:L_i\to B_j$ chosen uniform at random
from the construction of $\G$, define the random bipartite graph
$\Gamma_i(W;U,J)$ with vertex classes $U$, $J$ by the edge set
\[
 E\bigl(\Gamma_i\bigr)
 =\{(y,j)\in U\times J:f_{i,j}(y)\notin W\}.
\]
\end{definition}

Note that the condition $|W\cap L_i|\leq h$ in Definition \ref{des:random-bipartite} is needed, which ensures that $W\cap L_i$
can be extended to a set $U\in \binom{L_i}{h+s}$ with 
$|U\setminus S_i|\leq h$ for an admissible pair $(U,J)$. 

For a fixed $W$, define 
\[
I=\{i\in [v]:W\cap B_i=\varnothing,\ |W\cap L_i|\le h\}. 
\]
Note that a perfect matching in $\Gamma_i(W;U,J)$ with $i\in I$ and $W\cap L_i\subset U$ gives an edge $E_i(U,g)$ of $\G_i$ avoiding $W$.
If $W$ is a cover of $\G$, then such a $\Gamma_i(W;U,J)$ has no perfect matching. Then by the K\"onig--Hall
theorem, there is a nonempty set $U'\subseteq U$ whose neighborhood
$N(U')\subseteq J$ satisfies $|N(U')|\leq |U'|-1$.
Since $|J|=h+s$,
\[
 |J\setminus N(U')|\geq h+s+1-|U'|.
\]
Choose $J'\subseteq J\setminus N(U')$ with
$|J'|=h+s+1-|U'|\geq1$.
Since there are no edges between $U'$ and $J'$, by the definition of
$\Gamma_i(W;U,J)$ we have
\begin{equation}\label{des:hall}
 |U'|+|J'|=h+s+1\mbox{ and }
 f_{i,j}(U')\subseteq W\cap B_j
 \quad\text{for all }j\in J'.
\end{equation}

Fix some $W\subseteq V(\G)$ with $|W|<v$, and set
\[
 w_j=|W\cap B_j|,\, p=\sum_{1\leq j\leq v}(w_j-1)_+,\,
 I_0=\{j:W\cap B_j=\varnothing\}.
\]
\begin{claim}\label{des:available-indices}
\begin{equation}\label{des:available}
 |I|\ge |I_0|-|W\cap V(\mathcal L)|\ge p+1.
\end{equation}
\end{claim}

\begin{proof}
Since $V(\mathcal L),B_1,\ldots,B_v$ are pairwise disjoint,
\[
 |W|=\sum_{j=1}^v w_j+|W\cap V(\mathcal L)|
     =v-|I_0|+p+|W\cap V(\mathcal L)|\leq v-1.
\]
Rearranging gives $|I_0|-|W\cap V(\mathcal L) |\geq p+1$ in \eqref{des:available}.
On the other hand, $\Delta(\mathcal L)\leq h$ implies
\[
 \sum_{i=1}^v|W\cap L_i|
 =\sum_{y\in W\cap V(\mathcal L)}d_{\mathcal L}(y)\leq h|W\cap V(\mathcal L)|.
\]
Thus at most $|W\cap V(\mathcal L)|$ indices $i$ satisfy $|W\cap L_i|>h$.
Removing the indices with $|W\cap L_i|>h$ from $I_0$ gives
$|I|\geq |I_0|-|W\cap V(\mathcal L)|$ in \eqref{des:available}.
\end{proof}

\begin{lemma}\label{des:excess}
If $W$ is  a cover of $\G$ with $|W|<v$, then $p\ge h+s-1$.
\end{lemma}
\begin{proof}
Suppose $p<h+s-1$.
For any $i\in I$ and any admissible pair $(U,J)$ for $i$
with $W\cap L_i\subseteq U$, the graph $\Gamma_i(W;U,J)$
has no perfect matching, since $W$ covers $\G$.
Let $(U',J')$ be a witness satisfying \eqref{des:hall}.

For each $j\in J'$, by \eqref{des:hall} and the injectivity of
$f_{i,j}$ we have 
\[
 w_j=|W\cap B_j|\geq |f_{i,j}(U')|=|U'|.
\]
Hence, 
\[
 p=\sum_{j=1}^{v}(w_j-1)_+
 \geq\sum_{j\in J'}(w_j-1)_+
 \geq (|U'|-1)|J'|.
\]
If $|U'|\geq2$, then $2\leq|U'|\leq h+s$ and
$|U'|+|J'|=h+s+1$ imply
\[
 p\geq(|U'|-1)|J'|\geq h+s-1,
\]
a contradiction. Thus $|U'|=1$ and $J'=J$,
so every such graph $\Gamma_i(W;U,J)$
has an isolated vertex in $U$.

If $|I_0|\geq2$, choose $i\in I$ and
$j\in I_0\setminus\{i\}$, and choose an admissible pair
$(U,J)$ for $i$ with $j\in J$ and $W\cap L_i\subseteq U$.
Since $W\cap B_j=\varnothing$, the vertex $j$ is adjacent
to every vertex of $U$, contradicting the existence of
an isolated vertex in $U$. Thus we may assume $|I_0|=1$. 

By \eqref{des:available},  $p+|W\cap V(\mathcal L)|\leq |I_0|-1$. If $I_0=\{i\}$, then $p=0$ and $W\cap V(\mathcal L)=\varnothing$.
Choose two disjoint sets $U_1,U_2\subset L_i$ satisfying
$|U_a|=h+s$ and $|U_a\setminus S_i|\leq h$ for $a=1,2$: partition $S_i$ into
two $s$-sets and add disjoint $h$-subsets of $L_i\setminus S_i$.
Fix $j\ne i$ and choose each corresponding $J_a$ to contain $j$, $a=1,2$.
Each graph $\Gamma_i(W;U_a,J_a)$ has an isolated vertex in $U_a$, $a=1,2$.  Since $I_0=\{i\}$, we have $w_j\geq1$ for every $j\ne i$.
Moreover, $p=\sum_{\ell=1}^{v}(w_\ell-1)_+=0$, so $w_j=1$
for every $j\ne i$.
Both isolated vertices are nonadjacent to $j$, so their images
under $f_{i,j}$ lie in $W\cap B_j$. Since $w_j=1$ implies $|W\cap B_j|=1$,    their images
under $f_{i,j}$ equal the  unique
vertex in $W\cap B_j$, contradicting the injectivity of $f_{i,j}$.
\end{proof}

Recall that $|B_1|=|B_2|=\ldots=|B_v|=b$.

\begin{lemma}\label{des:probability}
For each ordered pair of distinct indices $i,j\in[v]$,
choose the injection $f_{i,j}:L_i\to B_j$
independently and uniformly at random. 
For fixed $W,i,U,J$ as in Definition~\ref{des:random-bipartite},
the probability that $\Gamma_i(W;U,J)$ has no perfect matching is at most
\begin{equation}\label{des:failure}
 e^{-(h+s)/4}+\sum_{j\in J}(w_j/b)^{h+s}.
\end{equation}
If $w_j\leq b/2$ for every $j\in J$, the probability that
$\Gamma_i(W;U,J)$ has no perfect matching is at most $e^{-(h+s)/8}$.
If $w_j=b-t>b/2$ for some $B_j$ and $t\geq1$, the probability of
having no perfect matching in $\Gamma_i(W;U,J)$ is at most
$e^{-(h+s)t/(4b)}$.
\end{lemma}
\begin{proof}
For a potential witness $(U',J')$ in \eqref{des:hall} with $|J'|=a$, we have $|U'|=h+s+1-a$. By the independence of the random injections $f_{i,j}$, $j\in J'$, for fixed $i$, 
\[
\begin{aligned}
\Pr((U',J') \mbox{ satisfies }\eqref{des:hall}) &\leq  \prod_{j\in J'}\Pr(f_{i,j}(U')\subseteq W\cap B_j)\\
  &\leq \prod_{j\in J'}\frac{(w_j)_{h+s+1-a}}{(b)_{h+s+1-a}}\\
 &\leq\prod_{j\in J'}(w_j/b)^{h+s+1-a}
 \leq\left(\frac{v}{ab}\right)^{a(h+s+1-a)}.
\end{aligned}
\]
Here $(w)_c=w(w-1)\ldots(w-c+1)$ is a falling factorial, and $\sum_jw_j<v$.
 If $J'=\{j\}$ with $j\in J$, then 
 \[
\begin{aligned}
\Pr((U',\{j\}) \mbox{ satisfies }\eqref{des:hall})\leq (w_j/b)^{h+s}.
\end{aligned}
\]
If $|J'|=1$ then $|U'|=h+s$, implying $U'=U$. Since there are at most $|J|$ potential witness $(U',J')$ with $|J'|=1$, we infer that 
\[
\Pr( \mbox{there exists some }(U',\{j\}) \mbox{ with $j\in J$ satisfies }\eqref{des:hall})\leq \sum_{j\in J}(w_j/b)^{h+s}.
\]
The witnesses $(U',J')$ with $|J'|=a\geq2$ contribute at most
\[
 \sum_{a=2}^{h+s}\binom{h+s}{h+s+1-a}\binom{h+s}{a}
           \left(\frac{v}{ab}\right)^{a(h+s-a+1)}
 \leq e^{-(h+s)/4}.
\]
Indeed, by \eqref{des:conditions} and \eqref{des:block-condition} 
\[
 7v/8\leq b<v,\qquad h+s\geq h\geq64\log v\geq64.
\]
For $2\leq a\leq\sqrt{h+s}$, 
\begin{align*}
 \log \binom{h+s}{a-1}\binom{h+s}{a}
           \left(\frac{v}{ab}\right)^{a(h+s+1-a)} &\leq 2a\log(h+s)-a(h+s-\sqrt{h+s})\log(2b/v)\\
           &\leq-(h+s)/2,
\end{align*}
using $2\log(h+s)\leq (h+s)/6$, $\sqrt{h+s}\leq (h+s)/8$, and
$\log(2b/v)\geq1/2$.
For $a>\sqrt{h+s}$, 
\begin{align*}
 \log \binom{h+s}{a-1}\binom{h+s}{a}
           \left(\frac{v}{ab}\right)^{a(h+s-a+1)} &\leq (h+s)\log4-(h+s)\log(\sqrt{h+s}\,b/v)\\
           &\leq-(h+s)/2.
\end{align*}
Summing at most $h+s$ terms gives $(h+s)e^{-(h+s)/2}\leq e^{-(h+s)/4}$.

If $w_j/b\leq1/2$ for every $j\in J$, then
by $\sum_{j\in J}w_j\leq|W|<v$ we have 
\[
\begin{aligned}
 \sum_{j\in J}(w_j/b)^{h+s}
 =\sum_{j\in J}\frac{w_j}{b}
       \left(\frac{w_j}{b}\right)^{h+s-1}
 \leq 2^{-(h+s-1)}\frac1b\sum_{j\in J}w_j
 \leq \frac vb\,2^{-(h+s-1)}.
 \end{aligned}
 \]
Since $v/b<2$, $\log2>1/2$ and $h+s\geq64$, we have
\[
 \frac vb\,2^{-(h+s-1)}
 \leq4e^{-(h+s)/2}
 \leq e^{-(h+s)/4}.
\]
Consequently,
\[
 e^{-(h+s)/4}+\frac vb\,2^{-(h+s-1)}
 \leq2e^{-(h+s)/4}
 \leq e^{-(h+s)/8},
\]
where we used $4\leq e^{(h+s)/4}$ and
$2\leq e^{(h+s)/8}$.

Now suppose that $w_j=b-t>b/2$ and $t\geq1$.
Then $w_i\leq v-b+t\leq2b/3$ for all $i\neq j$.
The right-hand side of \eqref{des:failure} is at most
$e^{-(h+s)t/b}+2(h+s) e^{-(h+s)/4}$, whether or not $j\in J$.
Since $1\leq t<b/2$ and $h+s\leq2b$,
\[
\begin{aligned}
 e^{-(h+s)t/(4b)}-e^{-(h+s)t/b}
 &\geq e^{-(h+s)/8}\min\{(h+s)t/(4b),1/2\}\\
 &\geq e^{-(h+s)/8}\frac{h+s}{4b}\\
 &\geq2(h+s)e^{-(h+s)/4}.
\end{aligned}
\]
The estimate
$e^{-(h+s)/8}(h+s)/(4b)\geq2(h+s)e^{-(h+s)/4}$
follows from $8b\leq8v\leq v^8\leq e^{(h+s)/8}$.
Thus the probability is at most $e^{-(h+s)t/(4b)}$.
\end{proof}

\subsection{With positive probability $\tau(\mathcal G)=v$}
\label{sec:covering}

\begin{proposition}\label{des:covering}
With probability at least $1/2$, $\tau(\G)=v$.
\end{proposition}
\begin{proof}
By Claim \ref{des:intersection-covers} we infer $\tau(\G)\le v$.
Since $\bigl|V(\mathcal L)\bigr|\leq v^2/4$ and $b<v$, we have $|V(\G)|=\bigl|V(\mathcal L)\bigr|+bv\leq2v^2$.
The number of sets $W\subseteq V(\G)$ of size less than $v$ is at most
\[
 \sum_{i=0}^{v-1}\binom{|V(\G)|}{i}
 \leq\left(\frac{e|V(\G)|}{v-1}\right)^{v-1}
 \leq e^{2v\log v}.
\]
For each fixed $W$, choose all admissible pairs used below before
exposing the injections. For distinct indices $i$, the chosen graphs
$\Gamma_i(W;U,J)$ depend on disjoint collections
$\{f_{i,\ell}:\ell\in[v]\setminus\{i\}\}$ of independent random
injections and hence are independent.

\begin{claim}\label{des:large-block}
With probability at least $1-v^{-2v}$, every cover
$W\subseteq V(\G)$ of $\G$
with $|W|<v$ satisfies $|W\cap B_j|>b/2$ for some $j$.
\end{claim}
\begin{proof}
Fix $W\subseteq V(\G)$ with $|W|<v$ and $w_j\le b/2$ for every $j$.
If $p<h+s-1$, then by Lemma~\ref{des:excess} $W$ is not a cover.
Otherwise, in view of Claim~\ref{des:available-indices} we have $|I|\ge p+1\ge h+s$.
For each $i\in I$, choose an admissible pair $(U,J)$ with
$W\cap L_i\subseteq U$. By Lemma~\ref{des:probability} and the
independence of the chosen graphs $\Gamma_i(W;U,J)$, $i\in I$, we have
\[
 \Pr(W\text{ covers }\G)\leq e^{-(h+s)^2/8}
       \leq e^{-4v\log v}.
\]
Here $h+s\geq h$ and $h^2\geq32v\log v$ are used. 
Since there are at most $e^{2v\log v}$ candidate sets, the probability
that any such $W$ covers $\G$ is at most
$e^{-2v\log v}=v^{-2v}$.
\end{proof}

By Claim~\ref{des:large-block}, it remains to exclude covers $W$
with $w_j>b/2$ for some $j$. Fix such an index $j$ and set
$t=b-w_j$, so $0\leq t<b/2$. Since $|W|\leq v-1$ and
$p\geq w_j-1$, \eqref{des:available} gives
\begin{equation}\label{des:outside}
 |W\setminus B_j|=|W|-w_j\leq v-b+t-1,\qquad
 |I|\geq p+1\geq w_j>b/2.
\end{equation}
For any $j,t$ with $0\leq t<b/2$,  define 
\[
\mathcal W_{j,t}=\{W\subset V(\G)\colon |W|<v,\ |W\cap B_j|=b-t \}.
\]
Then 
\begin{equation}\label{des:count-heavy}
\begin{aligned}
 |\mathcal W_{j,t}| &\leq \binom bt\sum_{a=0}^{v-b+t-1}\binom{|V(\G)\setminus B_j|}{a}\\
 &\leq b^t(|V(\G)|+1)^{v-b+t}\\
 &\leq v^{3(v-b)+4t}\leq v^{12(h+s)+4t}.
\end{aligned}
\end{equation}
Here we used $b<v$,
$|V(\G)|+1\leq2v^2+1\leq v^3$ and
$v-b<4h\leq4(h+s)$ from \eqref{des:block-condition}.

\begin{claim}\label{des:small-t-union}
The probability that there exist a cover $W\subseteq V(\G)$ of $\G$
with $|W|<v$ and an index $j\in[v]$ such that
$|W\cap B_j|>b/2$ and $|W\cap B_j|\geq b- m$
is at most $v^{-(h+s)}$.
\end{claim}
\begin{proof}
Fix $W,j$ with $|W|<v$, $w_j>b/2$ and $t=b-w_j\leq m$.
By \eqref{des:outside} and \eqref{des:block-condition},
\begin{align}\label{ineq-2.5-1}
|W\setminus B_j|\leq4h-16m+t-1<4h.
\end{align}
Exclude from $P_j$ the indices with $a_{ji}\in W$, $w_i>0$,
or $|W\cap L_i|>h$. The first two classes contain at most
$|W\cap L_j|$ and $\sum_{i\ne j}w_i$ indices, respectively,
since the vertices $a_{ji}$, $i\in P_j$, are distinct.
For the third class, $|L_i\cap L_j|\leq m$ implies
$|W\cap(L_i\setminus L_j)|>h-m$.
Counting pairs $(i,y)$ with such an $i$ and
$y\in W\cap(L_i\setminus L_j)$, using $\Delta(\mathcal L)\leq h$,
the number of indices $i\in P_j$ with $|W\cap L_i|>h$ is at most $|W\cap(V(\mathcal L)\setminus L_j)|\, h/(h-m)$.
Thus  the total number of excluded indices is at most
\[
\begin{aligned}
 &|W\cap L_j|+\sum_{i\ne j}w_i
       +\frac{h}{h-m}|W\cap(V(\mathcal L)\setminus L_j)|\\
 &\quad=|W\setminus B_j|
       +\frac{m}{h-m}|W\cap(V(\mathcal L)\setminus L_j)|\\
 &\quad\leq |W\setminus B_j|+\frac{4hm}{h-m}
       \leq |W\setminus B_j|+8m,
\end{aligned}
\]
where $|W\cap(V(\mathcal L)\setminus L_j)|\leq|W\setminus B_j|<4h$
and $h-m\geq h/2$ by \eqref{des:conditions}.
Since $|P_j|=4h$ and $a_{ij}=a_{ji}$, we obtain
\begin{equation}\label{des:good-partners}
\begin{aligned}
 \bigl|\{i\in P_j\cap I:a_{ij}\notin W\}\bigr|\geq4h-|W\setminus B_j|-8m \overset{\eqref{ineq-2.5-1}}{\geq}8m-t+1\geq4m.
\end{aligned}
\end{equation}

For each of the at least $4m$ indices in \eqref{des:good-partners},
choose $U$ containing $W\cap L_i$ and avoiding $a_{ij}$.
Such a choice with $|U|=h+s$ and $|U\setminus S_i|\leq h$
is possible because $|W\cap L_i|\leq h$, $|S_i|=2s$
and $|L_i\setminus S_i|=4h$.
Since $j\in P_i$ and $a_{ij}\notin U$, we have $j\notin D_i(U)$,
so $J$ can avoid $j$.
Then $w_\ell\leq|W\setminus B_j|<4h<b/2$ for every $\ell\in J$,
using \eqref{des:conditions} and \eqref{des:block-condition}.
By Lemma~\ref{des:probability} and the independence of the chosen
graphs $\Gamma_i(W;U,J)$ for distinct $i$,
\[
 \Pr(W\text{ covers }\G)
 \leq\bigl(e^{-(h+s)/8}\bigr)^{4m}=e^{-m(h+s)/2}.
\]
By \eqref{des:count-heavy} and the union bound,  
\[
\begin{aligned}
  \sum_{j\in [v]}\sum_{0\leq t\leq m} \sum_{W\in \mathcal{W}_{j,t}} \Pr(W\text{ covers }\G)&\leq v(m+1)v^{12(h+s)+4m}e^{-m(h+s)/2} \\
  &\leq v^{2+12(h+s)+4m-16(h+s)}\\
  &\leq v^{2-3(h+s)}\leq v^{-(h+s)},
\end{aligned}
\]
where we used $m+1\leq v$ from \eqref{des:conditions},
$m\geq8h^2/v\geq256\log v\geq32\log v$ and
$4m\leq h/4\leq h+s$.
\end{proof}

\begin{claim}\label{des:large-t-union}
The probability that there exist a cover $W\subseteq V(\G)$ of $\G$
with $|W|<v$ and an index $j\in[v]$ such that
$b/2<|W\cap B_j|<b-m$
is at most $v^{2-4(h+s)}$.
\end{claim}
\begin{proof}
Fix $W,j$ with $|W|<v$, $w_j>b/2$, and $t=b-w_j>m$.
Choose an admissible pair $(U,J)$ with $W\cap L_i\subseteq U$
for every $i\in I$.
By Lemma~\ref{des:probability}, the independence of the chosen
graphs $\Gamma_i(W;U,J)$, $i\in I$, and
$|I|>b/2$ from \eqref{des:outside},
\[
 \Pr(W\text{ covers }\G)
 \leq \left(e^{-(h+s)t/(4b)}\right)^{|I|}
 \leq e^{-(h+s)t/8}.
\]
For fixed $j,t$, by \eqref{des:count-heavy}  
\[
\sum_{W\in \mathcal{W}_{j,t}}\Pr(W\text{ covers }\G) \leq v^{12(h+s)+4t}e^{-(h+s)t/8}. 
\]
Since \eqref{des:conditions} implies
$4\log v\leq(h+s)/16$ and $t>m\geq256\log v$,
\[
\begin{aligned}
\log (v^{12(h+s)+4t}e^{-(h+s)t/8}) &=  (12(h+s)+4t)\log v-(h+s)t/8\\
 &\leq12(h+s)\log v-(h+s)t/16\\
 &\leq -4(h+s)\log v.
\end{aligned}
\]
Summing over at most $v^2$ choices of $j,t$, the probability
that any cover with $t>m$ exists is at most
\[
 \sum_{j\in [v]}\sum_{m\leq t\leq b}\sum_{W\in \mathcal{W}_{j,t}}\Pr(W\text{ covers }\G) \leq v^2v^{-4(h+s)}=v^{2-4(h+s)}. \qedhere
\]
\end{proof}
Thus by Claims \ref{des:large-block}, \ref{des:small-t-union} and \ref{des:large-t-union}, we conclude that
\[
\begin{aligned}
 \Pr\bigl(\tau(\G)<v\bigr) \leq v^{-2v}+ v^{-(h+s)}+v^{2-4(h+s)} \leq3v^{-2}<1/2. \qedhere
\end{aligned}
\]
\end{proof}

\begin{proof}[Proof of Theorem~\ref{main:many-covers}]
By Lemmas~\ref{des:parameter-bounds} and \ref{des:block-bounds},
the parameters $h,s,m,v$ in \eqref{des:choice} and $b$ in
\eqref{des:parameters} satisfy $0<b<v<k$ and \eqref{intro:cover-bounds}.
By Claim~\ref{des:intersection-covers}, $\G$ is an intersecting
$k$-graph. Proposition~\ref{des:covering} allows us to fix injections
with $\tau(\G)=v$.
By Claim~\ref{des:intersection-covers},
$\Rr\subseteq\hht^{(v)}(\G)$ and $T_{i,y}\in\hht^{(v)}(\G)$
for all $i\in[v]$ and $y\in L_i$.
Since the disjoint blocks $B_i$ all have size $b$,
we have $|\hht^{(v)}(\G)|\geq|\Rr|=b^v$.
Any cover of $\Rr$ contains some full $B_i$: otherwise one can
choose a vertex outside the cover from each $B_j$.
A cover of $\hht^{(v)}(\G)$ must additionally intersect the $4h+2s$ pairwise
disjoint sets $T_{i,y}$ outside that $B_i$.
Every cover of $\hht^{(v)}(\G)$ therefore has at least $b+4h+2s=k$ vertices.
Each edge of $\G$ covers $\hht^{(v)}(\G)$, so $\tau(\hht^{(v)}(\G))=k$.
\end{proof}

\section{Concluding remarks}

A more careful calculation in our construction gives the lower bound
below, while an optimization of Zakharov's proof~\cite{Z} gives the
upper bound: for all sufficiently large $k$,
\[
 k^k\exp\bigl(-10\sqrt{k\log k}\bigr)
 \leq m(k)
 \leq k^k\exp\bigl(-\sqrt{k}/5\bigr).
\]
Nevertheless, a factor of order $(\log k)^{1/2}$ remains between
the exponents of the two exponential correction terms.

\vspace{10pt}
{\noindent \bf Acknowledgment.} The second author was supported by the
National Natural Science Foundation of China (Grant No. 12622121).

\vspace{8pt}
{\bf\noindent AI declaration}: The construction was found using generative AI tools.
The authors verified the proofs and reworked the arguments
to clarify the mathematical ideas underlying the construction
and to provide a more conceptual exposition.

\end{document}